\documentclass{amsart}
\usepackage[utf8]{inputenc}
\usepackage{hyperref}
\usepackage{amsmath, amsthm, amssymb}
\usepackage{xcolor}
\usepackage{marvosym}
\usepackage{lmodern}
\usepackage{graphics}
\usepackage{hyperref}
\usepackage{setspace}
\usepackage[english]{babel}

\usepackage[
backend=biber,
style=nature,
maxnames=99,
sorting=nyt
]{biblatex}
\usepackage{nameref}
\usepackage{stmaryrd}
\usepackage{graphicx}
\newtheorem{theorem}{Theorem}

\newtheorem{corollary}{Corollary}

\newtheorem{proposition}{Proposition}
\newtheorem{lemma}{Lemma}

\newtheorem{remark}{Remark}

\title{The Nonlocal-to-Local Limit for the Inviscid Leray-$\alpha$ Equations}
\author{Jule Schindler}
\address{Department of Mathematics, Friedrich-Alexander-Universit\"at Erlangen-N\"urnberg, Cauerstr.~11, 91058 Erlangen, Germany}
\email{jule.schindler@fau.de}
\author{Emil Wiedemann}
\address{Department of Mathematics, Friedrich-Alexander-Universit\"at Erlangen-N\"urnberg, Cauerstr.~11, 91058 Erlangen, Germany}
\email{emil.wiedemann@fau.de}
\date{}
\begin{document}

\begin{abstract} We consider the inviscid Leray-$\alpha$ equations -- an inviscid nonlocal regularisation of the Euler equations. In the first part, we prove the convergence of strong solutions of the Leray-$\alpha$ equations to strong solutions of the Euler equations in $H^s(\mathbb{R}^d)$ for $s>d/2 +1 $, $d\in \{2,3\}$, for a large class of regularising kernels. In the second part, we consider weak solutions on a bounded domain with a local scaling property far away from the boundary. The scaling relates to second-order structure functions from turbulence theory and does not imply regularity. Nonetheless, under these assumptions, the weak solutions converge to (possibly wild) weak solutions of Euler in $L^2$ for almost every $t$.
\end{abstract}

\maketitle

\section{Introduction}

   We consider the \emph{incompressible inviscid  Leray-$\alpha$ equations} -- an inviscid  nonlocal regularisation of the incompressible Euler equations -- on $\Omega\times [0,T]$, $d=2,3$,
    \begin{align}
    \label{eq: model}
        \begin{cases}
        \partial_t v^{\alpha} + (u^{\alpha} \cdot \nabla) v^{\alpha} +\nabla p^{\alpha}=0 &\text{ on } \Omega\times [0,T], \\
        v^{\alpha}=u^{\alpha}-\alpha^2\Delta u^{\alpha}  &\text{ on } \Omega\times [0,T],\\
        \operatorname{div} v^{\alpha}= \operatorname{div} u^{\alpha}=0  &\text{ on } \Omega\times [0,T], \\
        v^{\alpha}(0,\cdot)  = v^{\alpha}_{0}  & \text{ on }  \Omega,
        \end{cases}
    \end{align} 
    with $\Omega\subseteq \mathbb{R}^d$ and where $v^{\alpha}:\Omega\times[0,T]\to \mathbb{R}^d$ is the velocity field, $p^{\alpha}:\Omega\times[0,T]\to\mathbb{R}$ is the pressure, and $\alpha>0$ is a given length scale.

The Leray-$\alpha$ model \eqref{eq: model} was introduced in \cite{Cheskidov.2005} and is named after Jean Leray, who smoothed the transport velocity in the nonlinear term in the Navier-Stokes equations to show the existence of weak solutions in \cite{Leray.1934} in 1934. Although similar so-called $\alpha$-turbulence models have been extensively studied in the literature, considering them either as regularisations of the Euler or Navier-Stokes equations or as large eddy simulation models of turbulence, comparatively little attention has been devoted to the inviscid Leray-$\alpha$ model \eqref{eq: model}. We here wish to answer fundamental questions about this inviscid regularisation of the Euler equations, which we feel to be particularly intuitive. Apart from the nonlocal-to-local limit ($\alpha\to0$) towards Euler, we are motivated by comparing the $\alpha$-limit with the vanishing viscosity limit of Navier-Stokes solutions.

If $\Omega$ is a bounded domain, we impose the boundary condition
\begin{equation}
    u^{\alpha}= 0 \quad \text{ on } \partial \Omega \times [0,T],
\end{equation}
which is chosen in analogy with the effect of viscosity. For a discussion of the boundary conditions in the case of the related Euler-$\alpha$ equations, see \cite{LopesFilho.2015}.

On the whole space, $\Omega=\mathbb{R}^d$, we can write more generally
\begin{equation}
\label{eq: general Kalpha}
    u^{\alpha}=K^{\alpha}*v^{\alpha},
\end{equation}
where $K^{\alpha}$ is associated with the Green function of the Helmholtz operator. To generalise certain results to a larger class of regularising kernels, we will remark consistently which properties of $K^{\alpha}$ have been used in each instance.

In Fourier variables, the relationship between $u^{\alpha}$ and $v^{\alpha}$ is written as
   \begin{equation}
   \label{eq: alpha kernel}
       \begin{aligned}
           \hat{u}^{\alpha}(\xi)=\frac{1}{1+\alpha^2|\xi|^2}\ \hat{v}^{\alpha}(\xi) \ \ \  \ \ \ \forall \xi \in \mathbb{R}^d.
       \end{aligned}
   \end{equation}

For $\alpha=0$, we formally obtain the \emph{incompressible Euler equations}
 \begin{align}\
 \label{eq: Euler}
        \begin{cases}
        \partial_t  v^{} + (v^{} \cdot \nabla) v^{} +\nabla p^{}=0  &\text{ on } \Omega\times [0,T], \\
        \operatorname{div} v^{}=0 &\text{ on } \Omega\times [0,T], \\
        v^{}(0,\cdot)  = v^{}_{0}   &\text{ on } \Omega.
        \end{cases}
    \end{align} 
    
For bounded domains $\Omega$, it is required additionally that
\begin{equation}
    v\cdot \hat{n} =0|_{\partial \Omega} \quad\text{ on } [0,T],
\end{equation}
where $\hat{n}$ is the unit exterior normal vector to $\partial \Omega$.\\

\textbf{Main Results.}
In the first part of this work, we show the following theorem.
    \begin{theorem}
    \label{thm:main2}
Let 
\begin{itemize}
\item $v_{0},v^{\alpha}_{0} \in H^s_{}(\mathbb{R}^d)$, $s>d/2+1$, divergence-free, 
\item $\|v_0 -v^{\alpha}_{0}\|_{H^s} \to 0$ as $\alpha \searrow 0$,
\item $T^*>0$ the time of existence of the solution of the Euler equations, i.e.\ such that for any $T<T^*$ there exists a unique solution $v^{} \in \mathcal{C}([0,T];H^s_{}(\mathbb{R}^d)) \cap W^{1,1}(0,T;H^{s-1}_{}(\mathbb{R}^d))$ with initial datum $v_0$. 
\end{itemize}
Then, for all $0<T<T^*$, there exists $0<\bar{\alpha} = \bar{\alpha}(v_0,v_{0}^{\alpha}, T)$ such that for all $\alpha \leq \bar{\alpha}$ there is a unique solution $v^{\alpha} \in \mathcal{C}([0,T];H^s_{}) \cap W^{1,1}(0,T;H^{s-1}_{})$ of the Leray-$\alpha$ equations with initial datum $v^{\alpha}_{0}$. 

Moreover,
   \begin{equation}
       \begin{aligned}
           \|v^{\alpha}-v\|_{L^{\infty}(0,T;H^s)} \to 0 \ \ \ \text{as} \  \alpha \searrow0,
       \end{aligned}
   \end{equation}
 and for all $0\leq t \leq T$, $s' \in [0,s-1]$,
    \begin{equation}
    \label{eq: thm main1 interpol}
        \begin{aligned}
            \|(v^{\alpha}-v)(t)\|_{H^{s'}} \leq C\left( \|v^{\alpha}_0-v_0\|_{H^{s'}} + \alpha^{\iota} t\right),
       \end{aligned}
    \end{equation}
where $C=C(\|v_0\|_{H^s},T)$ is independent of $\alpha$ and $\iota=s-s'$ for $s-2\leq s' \leq s-1$ and $\iota=2$ for $0\leq s'\leq s-2$.
    \end{theorem}
    
As a byproduct of the proof, we make the following observation.
      \begin{corollary}
      \label{corollary: main2}
Take the same hypotheses as in Theorem \ref{thm:main2}. Let 
\begin{itemize}
\item $K^{\alpha}$ such that for all $l\geq 0$, there exists $C>0$, independent of $\alpha$, such that $\|K^{\alpha}*\phi\|_{H^l} \leq C \|\phi\|_{H^l}$ for $\phi \in H^{l}$,
\item $\|K^{\alpha}*\phi-\phi\|_{H^{s}} \to 0$ as $\alpha \to 0$ for $\phi \in H^{s}$. 
\end{itemize}

Then, for all $0<T<T^*$, there exists $0<\bar{\alpha} = \bar{\alpha}(v_0,v_{0}^{\alpha}, T)$ such that for all $\alpha \leq \bar{\alpha}$ there is a unique solution $v^{\alpha} \in \mathcal{C}([0,T];H^s_{}) \cap W^{1,1}(0,T;H^{s-1}_{})$ of the Leray-$\alpha$ equations with kernel $K^{\alpha}$ (i.e.\ $u^{\alpha}=K^{\alpha}*v^{\alpha}$) and with initial datum $v^{\alpha}_{0}$. 

Moreover,
   \begin{equation}
       \begin{aligned}
           \|v^{\alpha}-v\|_{L^{\infty}(0,T;H^s)} \to 0 \ \ \ \text{as} \  \alpha \searrow0,
       \end{aligned}
   \end{equation}
   and for all $0\leq t \leq T$, $s' \in [0,s-1]$,
    \begin{equation}
        \begin{aligned}
            \|(v^{\alpha}-v)(t)\|_{H^{s'}} \leq C\left( \|v^{\alpha}_0-v_0\|_{H^{s'}} + \|(K^{\alpha}*v^{\alpha}-v^{\alpha})(t)\|_{H^{s'}} t\right),
       \end{aligned}
    \end{equation}
where $C=C(\|v_0\|_{H^s},T)$ is independent of $\alpha$.
  \end{corollary} 

\begin{remark}
    The assumptions on $K^{\alpha}$ in Corollary~\ref{corollary: main2} are automatically fulfilled by a family of approximations of the identity generated by an $L^1$ function of unit integral.
\end{remark}

The convergence in $H^s$ is proved adopting the idea from \cite{DaBeiraoVeiga.1994b, DaBeiraoVeiga.1994,MR1199199} of regularising the initial data. This concept has also been used in Masmoudi's work \cite{Masmoudi.2007}, where the inviscid limit of Navier-Stokes solutions is considered. A corresponding statement has been shown by Linshiz and Titi in \cite{Linshiz.2010} for the Euler-$\alpha$ model
\begin{equation}
\label{eq: euler-alpha}
\begin{cases}
    \partial_t v^{\alpha} + (u^{\alpha}\cdot\nabla) v^{\alpha} + v^{\alpha}_j \nabla u^{\alpha}_j + \nabla p^{\alpha}= 0,\\
    v^{\alpha}=u^{\alpha}-\alpha^2 \Delta u^{\alpha},\\
    \operatorname{div}v^{\alpha}= \operatorname{div}u^{\alpha}=0,\\
     v^{\alpha}(x,0)=v_0(x),
\end{cases}
\end{equation}
introduced in \cite{Holm.1998, Holm.1998b} and which turns out to be the zero-viscosity version of visco-elastic second-grade non-Newtonian complex fluid equations (see e.g.\ \cite{Dunn.1974}). However, the proof in \cite{Linshiz.2010} depends heavily on the vorticity structure of the system: In $\mathbb{R}^3$ we have
$$ \partial_t q^{\alpha} + (u^{\alpha}\cdot \nabla) q^{\alpha}=(q^{\alpha}\cdot \nabla) u^{\alpha}, $$
and in $\mathbb{R}^2$
$$ \partial_t q^{\alpha} + (u^{\alpha}\cdot \nabla) q^{\alpha}=0, $$
where $q^{\alpha}=\operatorname{curl} v^{\alpha}$. 
We conduct our proofs without using the vorticity. Moreover, we do not restrict the calculations to the $\alpha$-kernel from \eqref{eq: alpha kernel} until the end of the proof and hence allow to investigate the convergence and its rate for a more general kernel $K^{\alpha}$. 

Comparing the inviscid limit in \cite{Masmoudi.2007} (for $\nu=\alpha^2$) and the Euler-$\alpha$ limit in \cite{Linshiz.2010} with the Leray-$\alpha$ limit here, there is no difference between the convergence rates regarding the order of $\alpha$.

It is not surprising -- due to the vorticity formulation -- that the Euler-$\alpha$ model has gained more interest in the literature than the inviscid Leray-$\alpha$ equations. Apart from \cite{Linshiz.2010}, there is ample literature about the $\alpha$-convergence of the Euler-$\alpha$ equations: e.g.\ \cite{Abbate.2024} on the two-dimensional torus, \cite{Busuioc.2017,LopesFilho.2015} on a two-dimensional bounded domain with Dirichlet boundary conditions, \cite{Bardos.2010} considering vortex sheet initial data, \cite{Busuioc.2020} on the half-plane with no-slip boundary conditions  -- all of them making use of the vorticity structure of the two-dimensional system. In particular in \cite{LopesFilho.2015}, the difference between the viscosity limit of Navier-Stokes and the $\alpha$-limit of Euler-$\alpha$ regarding the boundary layer is highlighted. The authors did not end up with a Kato-like criterion as expected but overcame the boundary layer under appropriate regularity assumptions. 

To our knowledge, there has been no contribution in the literature about the $\alpha$-limit of the inviscid Leray-$\alpha$ equations. In \cite{Boutros.2023}, Onsager’s conjecture for the inviscid Leray-$\alpha$ model and other inviscid $\alpha$-models is studied. For comparisons with other $\alpha$-turbulence models see also \cite{Guermond.2003, Holm.1999}. In \cite{Cheskidov.2005}, the authors investigate the dimension of the global attractor and the energy spectrum of the viscous Leray-$\alpha$ model. The convergence of the Leray-$\alpha$ model to the Navier-Stokes equations is investigated in \cite{Cao.2009} and \cite{V.Chepyzhov.2007}.
In the context of sub-grid scale models, the Leray-$\alpha$ model is treated in \cite{Geurts.2006, Geurts.2008, PietarilaGraham.2008} and is considered as successful large eddy simulation model of turbulence for channel and pipe flows \cite{Cheskidov.2005}. 

Besides the already mentioned literature, the Euler-$\alpha$ model is treated in \cite{Busuioc.1999, Busuioc.2003, MARSDEN.2003,Marsden.2000,Shkoller.2000}. 
In \cite{Oliver.2001}, the global existence of unique weak solutions with initial vorticity in the space of Radon measures on $\mathbb{R}^2$ is derived. In \cite{Vorotnikov.2012}, dissipative (very weak) solutions of Euler-$\alpha$ are defined and their existence is shown.  
Paper \cite{Busuioc.2012} deals with the simultaneous $\alpha$- and viscosity limit of second-grade fluid equations in a bounded domain with Navier-type boundary conditions, whereas \cite{MR3345360} studies Dirichlet boundary conditions.

In the second part, we assume the initial data only in $L^2$, thus arguably allowing for turbulent behaviour. To conclude convergence of weak solutions of the Leray-$\alpha$ equations, we assume a local scaling property. This idea was used in \cite{Constantin.2018} and then in \cite{Drivas} to show convergence of a vanishing viscosity sequence of Navier-Stokes solutions to weak solutions of the Euler equations. The condition is considered far away from boundaries and problems with expected boundary layers are avoided.

The scaling property, which will be defined in \eqref{eq: scaling prop}, relates to second-order structure functions from Kolmogorov's theory of turbulence from 1941 (see \cite{Frisch.1995}). The theory suggests that in three-dimensional turbulent flow, energy is (on average) transferred from larger to smaller scales until it is dissipated through viscous effects. The range of scales where eddies disaggregate and where the dynamics, described by the Navier-Stokes equations, is dominated by the nonlinear term is called inertial range. The scaling of the second-order structure function was originally inferred from the self-similarity hypothesis of turbulent flow in the inertial range and is supported by experimental results. Due to the fact that the scaling property in \eqref{eq: scaling prop} is only assumed up to $\eta>0$ (which would correspond to the Kolmogorov dissipation scale in the case of viscosity), our assumption does not imply any regularity. 

A second important statement of turbulence theory (Kolmogorov's `Zero-th Law') is that turbulent flow is dissipative, i.e.\ it does not conserve energy as strong solutions of Euler or of Leray\mbox{-}$\alpha$ do. Note that Theorem \ref{thm: 2} allows for solutions of Euler that display such anomalous dissipation of energy.

In Section~\ref{subsection: scaling limit}, we show
\begin{theorem}
\label{thm: 2}
    Let $(v^{\alpha})_{\alpha>0}$ be a sequence of weak solutions of the inviscid Leray-$\alpha$ equations as in Proposition \ref{prop: weak sol}
with $ v^{\alpha}(0) \in H(\Omega)$. Let  $\eta(\alpha)$ be such that $\eta(\alpha)>0$ for $\alpha>0$ and $\eta(\alpha)\to0$ as $\alpha \to 0$. We assume that for any $K\subset \subset \Omega$ there exists a constant $E_K>0$ and a constant $\gamma>0$ such that
\begin{equation}
\label{eq: scaling prop}
    \underset{\alpha > 0}{\operatorname{sup}} \int_0^T \int_K |v^{\alpha}(x+y,t)-v^{\alpha}(x,t)|^2 \ dx \ dt \leq E_K |y|^{2\gamma}
\end{equation}
for $|y|< dist(K,\partial \Omega)$ and 
$|y| \geq \eta(\alpha)$. We assume that $v^{\alpha}(t)$ converge weakly in $L^2(\Omega)$ to $v^{\infty}(t)$ for almost all $t\in (0,T)$. Then $v^{\infty}$ is a weak solution of the Euler equations.
\end{theorem}
\begin{remark}
The limit solution $v^{\infty}$ inherits the scaling property \eqref{eq: scaling prop}.
\end{remark}

\section{Preliminaries}
Throughout the paper, $C$ will be used as a generic constant. Relevant dependencies will be noted as a subscript.

 \begin{lemma}
 \label{lemma: relation norms}
Let $v\in H^s(\mathbb{R}^n)$ for $s\in \mathbb{R}$ and $u=(\operatorname{Id}-\alpha^2\Delta)^{-1}v$ for $\alpha\in(0,1]$. Then,
$$ \|u\|_{H^s} \leq \|v\|_{H^s}, \ \ \ \|u\|_{H^{s+1}} \leq \frac{1}{\alpha} \|v\|_{H^s} \ \ \ \text{ and } \ \ \ \|u\|_{H^{s+2}} \leq \frac{1}{\alpha^2} \|v\|_{H^s}.$$
\end{lemma}
The proof can be found below in \hyperref[subsection: relation between norms]{A1}.

\begin{proposition}[Existence of strong solutions for Leray-$\alpha$] \label{thm:strong sol}
Let $v^{\alpha}_{0} \in H^1_{}(\mathbb{R}^d)$ and divergence-free. Then, there exists $T_{\alpha}^*=T_{\alpha}^*(\|v^{\alpha}_{0}\|_{H^1},\alpha), \ T_{\alpha}^*\geq \frac{C_{\alpha}}{\|v^{\alpha}_{0}\|_{H^1}}>0$, such that for any $T<T_{\alpha}^*$ there exists a unique solution $v^{\alpha} \in \mathcal{C}([0,T];H^1_{}(\mathbb{R}^d))\cap W^{1,1}(0,T;L^2_{}(\mathbb{R}^d))$ of the Leray-$\alpha$ equations \eqref{eq: model} with initial datum $v^{\alpha}_{0}$.

Let $v^{\alpha}_{0} \in H^s_{}(\mathbb{R}^d)$, $s>d/2+1$, and divergence-free. Then, there exists $T^*=T^*(\|v^{\alpha}_{0}\|_{H^s}), \ T^*\geq \frac{C}{\|v^{\alpha}_{0}\|_{H^s}}>0$, such that for any $T<T^*$ there exists a unique solution $v^{\alpha} \in \mathcal{C}([0,T];H^s_{}(\mathbb{R}^d)) \cap W^{1,1}(0,T;H^{s-1}_{}(\mathbb{R}^d))$ uniformly in $\alpha$.
 \end{proposition}
The existence of strong solutions for Leray-$\alpha$ can be shown by similar techniques as used in the proofs for the Euler equations.
\begin{proposition}[Existence of strong solutions for Euler (e.g.\ \cite{Bedrossian.2022, Kato.1972, Majda.2010})] \label{thm:strong sol Euler}
Let $v^{}_{0} \in H^s_{}(\mathbb{R}^d)$, $s>d/2+1$, and $\operatorname{div}v_0=0$. Then, there exists $T^*=T^*(\|v^{}_{0}\|_{H^s}), \ T^*\geq \frac{C}{\|v^{}_{0}\|_{H^s}}>0$, such that for any $T<T^*$ there exists a unique solution $v^{} \in \mathcal{C}([0,T];H^s_{}(\mathbb{R}^d)) \cap W^{1,1}(0,T;H^{s-1}_{}(\mathbb{R}^d))$ of the Euler equations \eqref{eq: Euler} with initial datum $v_0$.
\end{proposition}

 \begin{remark}
 \label{remark: assumptions u for unif existence}
The first part of Proposition \ref{thm:strong sol} crucially depends on the properties of the kernel $K^{\alpha}$ where $u^{\alpha}=K^{\alpha}*v^{\alpha}$. We need that
    $$ \|K^{\alpha}*v^{\alpha}\|_{H^{s+d/2+}} \leq C_{\alpha} \|v^{\alpha}\|_{H^s},$$
where we have used the notation that $v\in H^{l+}$ if there exists a $\delta>0$ such that $v\in H^{l+\delta}$.

For the second part, it is enough to have
$$ \|K^{\alpha}*v^{\alpha}\|_{H^s} \leq C \|v^{\alpha}\|_{H^s} $$
for $s\geq 0$ and $C$ independent of $\alpha$.
 \end{remark}
 \begin{remark}
 \label{remark: global in time unif sol}
We can easily obtain global-in-time solutions even for $d=3$ by improving the regularisation kernel $K^{\alpha}$.
\end{remark}

Let
$$ H(\Omega):= \{v\in L^2(\Omega): \operatorname{div}v=0 \text{ in } \Omega,\quad v \cdot n = 0 \text{ on } \partial \Omega\},$$
$$V(\Omega):= \{ v\in H_0^1(\Omega): \operatorname{div}v=0\text{ in } \Omega\} = H(\Omega) \cap H_0^1(\Omega).$$
We call $v^{\alpha}\in L^{\infty}(0,T;H(\Omega))$ a \emph{weak solution} of the Leray-$\alpha$ equations \eqref{eq: model} if it satisfies
$$  \int_0^T \int_{\Omega}  v^{\alpha}(t,x) \partial_t\Phi(t,x) + (u^{\alpha}(t,x) \otimes v^{\alpha}(t,x)) : \nabla \Phi(t,x) \ dx \ dt  $$ $$= \int_{\Omega}  v^{\alpha}(T,x) \Phi(T,x) -  v^{\alpha}_0(x) \Phi(0,x) \ dx$$
for any divergence-free test function $\Phi \in C^{\infty}_c(\Omega\times [0,T] )$, with 
\begin{equation*}
 v^{\alpha}(x,t)=(\operatorname{Id}-\alpha^2 \Delta)u^{\alpha}(x,t)
\end{equation*} 
and $u^{\alpha}|_{\partial \Omega}=0$.
The existence of weak solutions of Leray-$\alpha$ on a smooth bounded domain can be shown with the Galerkin method where the approximate solutions belong to the finite-dimensional space spanned by the first eigenfunctions of the Stokes operator. 
\begin{proposition}[Existence of weak solutions of Leray-$\alpha$]
\label{prop: weak sol}
    Let $\Omega \subset \mathbb{R}^d$, $d\in \{2,3\}$, be a smooth bounded domain. Let $v_0 \in H(\Omega)$ and $T>0$. Then there exists a weak solution $v^{\alpha}$ to the Leray-$\alpha$ equations such that
    $$ v^{\alpha} \in L^{\infty}(0,T; H(\Omega)) \cap C([0,T]; V'),$$
and where
$$ u^{\alpha} \in L^{\infty}(0,T; V(\Omega)\cap H^2(\Omega))$$
satisfies
$$ v^{\alpha}(x,t)=(\operatorname{Id}-\alpha^2 \Delta)u^{\alpha}(x,t).$$
Moreover,
\begin{equation}\label{admissible}
\| v^{\alpha}\|_{L^{\infty}(0,T;H(\Omega))} \leq \|v^{\alpha}(0)\|_{H(\Omega)}. 
\end{equation}
\end{proposition}

For the Euler equations $(\alpha=0)$, the existence of admissible weak solutions (i.e. such that satisfy an energy inequality like~\eqref{admissible}) for any initial data in $H$ is still open. Non-admissible weak solutions have been produced in~\cite{wiedemann2011}.

\section{Convergence of strong solutions in $H^s$}  
In this section, we show Theorem \ref{thm:main2} and Corollary \ref{corollary: main2}, i.e.\ the convergence of strong solutions of the Leray-$\alpha$ equations to strong solutions of the Euler equations on the whole space. The proof is based on Masmoudi's work in \cite{Masmoudi.2007} where the inviscid limit of solutions of the Navier-Stokes equations is considered. Berselli and Bisconti have used these techniques to show the convergence of solutions of the Euler-Voigt equations to solutions of Euler in \cite{Berselli.2012}. Similar calculations have been done by Linshiz and Titi in \cite{Linshiz.2010} for the $\alpha$-limit of the Euler-$\alpha$ equations. While the authors in \cite{Linshiz.2010} utilise the vorticity structure of the Euler-$\alpha$ model to handle the additional term in the equation compared with the Leray-$\alpha$ model, here, we do not need to make use of the vorticity.

The proof is divided into two parts. First, the strong convergence in $H^s$ is proven. Secondly, we show \eqref{eq: thm main1 interpol}. Corollary \ref{corollary: main2} is a byproduct, as we keep the kernel general, as far as possible, in our calculations. Specifically, until further notice, we assume only that for all $l \geq 0$ there exists $C>0$ such that
\begin{equation}
\label{eq: kernel assumption}
    \|K^{\alpha}*v\|_{H^l} \leq C \|v\|_{H^l} 
\end{equation}
with $C$ independent of $\alpha$. Note that the existence result for $v_0^{\alpha}\in H^s$ in Proposition~\ref{thm:strong sol} remains true (Remark \ref{remark: assumptions u for unif existence}).

We are going to use the following estimates for divergence-free fields for $s\in \mathbb{R}$:
\begin{equation}
\label{eq: est 1}
    \|(u\cdot \nabla) v\|_{H^s} \leq C (\|u\|_{H^s} \|v\|_{H^s} + \|u\|_{L^{\infty}} \|v\|_{H^{s+1}}), \text{ for $s>d/2+1$,}
\end{equation}
\begin{equation}\label{eq: est 0}
 |\langle (u\cdot \nabla)v,v \rangle_{H^s} |\leq C \|\nabla u\|_{H^s} \|v\|_{H^s}^2, \text{ for $s>d/2$,}
\end{equation}
\begin{equation}
    \label{eq: est 2}
    |\langle (u\cdot \nabla)v,v \rangle_{H^s} | \leq C \|u\|_{H^s} \|v\|_{H^s}^2,\text{ for $s>d/2+1$,}
\end{equation}
and
\begin{equation}
\begin{aligned}
    \label{eq: est 4}
    &|\langle (u\cdot \nabla)v,v \rangle_{H^s} |\\
		&\quad\leq C (\|u\|_{H^l} \|v\|_{H^s} + \|v\|_{H^l} \|u\|_{H^s}) \|v\|_{H^s}, \text{ for $s\geq 0,\quad l>d/2+1$.}
		\end{aligned}
\end{equation}
A proof of \eqref{eq: est 1} can be found in \cite[Proposition A.29]{Bedrossian.2022}. Estimate~\eqref{eq: est 0} is Corollary 2.1 in~\cite{Feff.2014}. Estimate \eqref{eq: est 2} is a special case ($s=l$) of \eqref{eq: est 4}, which follows from the commutator estimate in \cite{Kato.1988}.

\begin{proof}[Proof of Theorem 1]
Let 
\begin{equation*}v^{\delta}_0 = \mathcal{F}^{-1}\left(\chi_{\{|\xi|\leq 1/\delta\}}(\xi) \mathcal{F}(v_0)\right) \quad\text{and}\quad v^{\alpha,\delta}_0 = \mathcal{F}^{-1}\left(\chi_{\{|\xi|\leq 1/\delta\}}(\xi) \mathcal{F}(v^{\alpha}_0)\right)
\end{equation*}
for $\delta>0$. It follows
\begin{equation}
    \label{eq: x}
    \|v_0^{\delta}\|_{H^s} \leq \|v_0\|_{H^s}, \ \  \ \ \ \ \ \ \|v_0^{\delta}\|_{H^{s+1}} \leq \delta^{-1} \|v_0\|_{H^s}
\end{equation}
and for $l\in[0,s]$
\begin{equation}
    \label{eq: y}
    \|v_0^{\delta}-v_0\|_{H^l} \leq \delta^{s-l} \|v_0\|_{H^{s}}.
\end{equation}
The same estimates hold for $v_0^\delta$ and $v_0$ replaced by $v_0^{\alpha,\delta}$ and $v_0^\alpha$, respectively.

Let $v^{\delta}$ be the solution of the Euler equations for initial data $v_0^{\delta}$. Let $v^{\alpha,\delta}$ be the solution of the Leray-$\alpha$ equations for initial data $v_0^{\alpha,\delta}$. We will show that, for any $s>d/2+1$ and $\epsilon>0$, we find first a $\delta>0$ and then $\alpha =\alpha_\delta>0$ such that
\begin{equation}\label{eq: to show}
\begin{aligned}
 \|v^{\alpha} - v \|_{L^{\infty}(0,T;H^s)}\leq &\|v^{\alpha} - v^{\delta,\alpha}\|_{L^{\infty}(0,T;H^s)}\\
& + \|v^{\delta,\alpha} - v^{\delta}\|_{L^{\infty}(0,T;H^s)} + \|v^{\delta}- v \|_{L^{\infty}(0,T;H^s)}\\
		&\quad< \epsilon.
		\end{aligned}
\end{equation}
First, notice that $v, \ v^{\alpha}, \ v^{\delta}, \ v^{\delta,\alpha}$ exist on a time interval $[0,T_0]$ which is independent of $\alpha$ and $\delta$: Indeed, let $(\phi,\xi) \in \{(v,v), (v^{\delta},v^{\delta}), (v^{\alpha}, u^{\alpha}), (v^{\delta,\alpha}, u^{\delta,\alpha})\}$ where $u^{\alpha}:=K^{\alpha}*v^{\alpha}$ and $u^{\delta,\alpha}:=K^{\alpha}*v^{\delta,\alpha}$. We obtain from
$\partial_t \phi + (\xi \cdot \nabla) \phi = -\nabla p$
that, for $s>d/2+1$, by \eqref{eq: est 2}, \eqref{eq: est 4} and by Lemma \ref{lemma: relation norms} (or for a kernel such that \eqref{eq: kernel assumption})
\begin{equation}
    \label{eq: H^s estimate}
    \frac{d}{dt} \|\phi\|^2_{H^s} \leq C\|\xi\|_{H^s} \|\phi\|_{H^s}^2 \leq  C\|\phi\|_{H^s}^3,
\end{equation}
\begin{equation}
    \label{eq: H^s+1 estimate}
    \frac{d}{dt} \|\phi\|^2_{H^{s+1}} \leq C (\|\xi\|_{H^s} \|\phi\|_{H^{s+1}}^2 + \|\phi\|_{H^s} \|\xi\|_{H^{s+1}} \|\phi\|_{H^{s+1}}) \leq C \|\phi\|_{H^s} \|\phi\|_{H^{s+1}}^2.
\end{equation}
Let $\alpha^*$ be such that $v_0^{\alpha^*}\neq 0$. Since $\|v_0^{\alpha}-v_0\|_{H^s} \to 0$ as $\alpha \to 0$, there exists $\alpha_0$ such that 
\begin{equation}
\label{eq: v_0^a indep of alpha}
    \|v_0^{\alpha}\|_{H^s} \leq \|v_0\|_{H^s} + \|v_0^{\alpha^*}\|_{H^s}
\end{equation}
for all $\alpha \leq \alpha_0$. Since $\|\phi(0)\|_{H^s} \leq C(\|v_0\|_{H^s}, \|v_0^{\alpha}\|_{H_s}) \leq C(\|v_0\|_{H^s}) \leq K$ (by  \eqref{eq: x} and \eqref{eq: v_0^a indep of alpha}), we can find a $T_0>\frac{C}{K}>0$, which is independent of $\alpha \leq \alpha_0$ and $\delta$, such that the $H^s$ norm of $\phi$ stays bounded during this time interval:
\begin{equation}
\label{eq: phi leq K}
    \|\phi(t)\|_{H^s} \leq C(T_0,K),
\end{equation}
and consequently by equations \eqref{eq: x} and \eqref{eq: H^s+1 estimate}
\begin{equation}
\label{eq: abc}
\begin{aligned}
    \|\phi^{\delta}(t)\|_{H^{s+1}} &\leq \|\phi^{\delta}(0)\|_{H^{s+1}} \ \operatorname{exp}\left(\int_0^t\|\phi^{\delta}(\tau)\|_{H^{s}}d\tau\right) \\
		&\leq  \frac{K}{\delta} \ \operatorname{exp}(C(T_0,K)) \leq \frac{C(T_0,K)}{\delta}
		\end{aligned}
\end{equation}
where $\phi^{\delta} \in \{v^{\delta}, v^{\delta,\alpha}\}$.\\

Now, we have all the ingredients to prove \eqref{eq: to show}. Let $C$ be a generic constant possibly depending on $K$ and $T_0$.\\

(a) Let $\frac d2<\sigma<s-1$. Setting $w^{\delta,\alpha}=v^{\delta, \alpha}-v^{\alpha}$, we have
\begin{equation*}
 \partial_tw^{\delta,\alpha} + (K^{\alpha}*w^{\delta,\alpha} \cdot \nabla)v^{\delta,\alpha} + (K^{\alpha}*v^{\alpha} \cdot \nabla) w^{\delta,\alpha} = -\nabla (p^{\delta,\alpha}-p^{ \alpha}).
\end{equation*}

Taking the $H^s$-inner product of this equation with $w^{\delta,\alpha}(t)$, by Lemma~\ref{lemma: relation norms} (or assumption~\eqref{eq: kernel assumption}) and equations \eqref{eq: est 1} and \eqref{eq: est 2} together with the Sobolev embedding $H^\sigma\subset L^\infty$,
\begin{equation}
\label{sigmaest}
\begin{aligned}
& \partial_t \|w^{\delta,\alpha} \|_{H^s}^2 \\ &\leq C(|\langle (K^{\alpha}*w^{\delta,\alpha} \cdot \nabla)v^{\delta,\alpha}, w^{\delta,\alpha} \rangle_{H^s}| + |\langle (K^{\alpha}*v^{\alpha} \cdot \nabla) w^{\delta,\alpha}, w^{\delta,\alpha}\rangle_{H^s}|)\\
&\leq C (\|w^{\delta,\alpha}\|_{H^{\sigma}} \| v^{\delta,\alpha}\|_{H^{s+1}}\|w^{\delta,\alpha}\|_{H^s}) +   C(\|v^{\alpha}\|_{H^s} + \|v^{\delta,\alpha}\|_{H^s} ) \|w^{\delta,\alpha}\|_{H^s}^2.
\end{aligned}
\end{equation}
Note that $ \| v^{\delta,\alpha}\|_{H^{s+1}} \leq \frac{C}{\delta}$ by~\eqref{eq: abc}. To compensate for the factor $\delta^{-1}$, we estimate, using~\eqref{eq: est 0} and $\sigma+1<s$, 
\begin{equation*} 
\begin{aligned}
\partial_t \|w^{\delta,\alpha} \|_{H^\sigma}^2 &\leq C(|\langle (K^{\alpha}*w^{\delta,\alpha} \cdot \nabla)v^{\delta,\alpha}, w^{\delta,\alpha} \rangle_{H^\sigma}| + |\langle (K^{\alpha}*v^{\alpha} \cdot \nabla) w^{\delta,\alpha}, w^{\delta,\alpha}\rangle_{H^\sigma}|)\\
&\leq C (\|w^{\delta,\alpha}\|_{H^{\sigma}} \| v^{\delta,\alpha}\|_{H^{\sigma+1}}\|w^{\delta,\alpha}\|_{H^\sigma}) +  C \|K^{\alpha}*\nabla v^{\alpha}\|_{H^\sigma}  \|w^{\delta,\alpha}\|_{H^\sigma}^2\\
&\leq  C (\|v^{\delta,\alpha}\|_{H^s}+\|v^{\alpha}\|_{H^s})  \|w^{\delta,\alpha}\|_{H^\sigma}^2.
\end{aligned}
\end{equation*}
By~\eqref{eq: phi leq K}, $\|v^{\delta,\alpha}\|_{H^s}+\|v^{\alpha}\|_{H^s}$ is bounded in time, hence from~\eqref{eq: y} it follows that 
\begin{equation*}
 \|w^{\delta,\alpha} \|_{H^\sigma}\leq C \delta^{s-\sigma}.
\end{equation*} 
Inserting this back into~\eqref{sigmaest} and using Gronwall, we arrive at

\begin{equation*}
 \|v^{\delta, \alpha}-v^{\alpha}\|_{L^{\infty}(0,T_0;H^s)} \leq C(\|v_0^{\delta,\alpha}-v_0^\alpha\|_{H^s}+\delta^{s-\sigma-1}T_0).
\end{equation*}

(b) Completely analogous arguments yield 

\begin{equation*}
 \|v^{\delta}-v\|_{L^{\infty}(0,T_0;H^s)} \leq C(\|v_0^{\delta}-v_0\|_{H^s}+\delta^{s-\sigma-1}T_0).
\end{equation*}

(c) Setting $w^{\delta}=v^{\delta, \alpha}-v^{\delta}$,
$$ \partial_t w^{\delta} + (K^{\alpha}*w^{\delta} \cdot \nabla) v^{\delta, \alpha} + (K^{\alpha}*v^{\delta} \cdot \nabla) w^{\delta}+  ((K^{\alpha}*v^{\delta}-v^{\delta}) \cdot \nabla) v^{\delta} = -\nabla (p^{\delta,\alpha}-p^{\delta}).$$
Taking the $H^s$ inner product with $w^{\delta}(t)$,
\begin{equation*}
\begin{aligned}
&\partial_t \|w^{\delta} \|_{H^s}^2 \leq C(|\langle (K^{\alpha}*w^{\delta} \cdot \nabla) v^{\delta, \alpha}, w^{\delta}\rangle_{H^s}| + |\langle (K^{\alpha}*v^{\delta} \cdot \nabla) w^{\delta}, w^{\delta}\rangle_{H^s}|\\
&\quad\quad+  |\langle ((K^{\alpha}*v^{\delta}-v^{\delta}) \cdot \nabla) v^{\delta},w^{\delta} \rangle_{H^s}|)\\
&\leq C \|w^{\delta}\|_{L^\infty} \| v^{\delta, \alpha}\|_{H^{s+1}} \|w^{\delta}\|_{H^s} + C(\|v^{\delta}\|_{H^s} + \|v^{\delta,\alpha}\|_{H^s}) \|w^{\delta}\|_{H^s}^2\\ 
&+C(\|K^{\alpha}*v^{\delta}-v^{\delta}\|_{L^\infty} \|v^{\delta}\|_{H^{s+1}} +  C\|K^{\alpha}*v^{\delta}-v^{\delta}\|_{H^s} \|v^{\delta}\|_{H^{s}}) \|w^{\delta}\|_{H^s}. 
\end{aligned}
\end{equation*}
On the one hand, $ \| v^{\delta,\alpha}\|_{H^{s+1}}  \leq \frac{C}{\delta}$ and $ \| v^{\delta}\|_{H^{s+1}} \leq \frac{C}{\delta}$ by \eqref{eq: abc}. On the other hand, for $d/2<\sigma<  s-1$, we can estimate similarly as in (a) as follows:
\begin{equation*}
\begin{aligned}
&\partial_t \|w^{\delta} \|_{H^\sigma}^2 \leq C(|\langle (K^{\alpha}*w^{\delta} \cdot \nabla) v^{\delta, \alpha}, w^{\delta}\rangle_{H^\sigma}| + |\langle (K^{\alpha}*v^{\delta} \cdot \nabla) w^{\delta}, w^{\delta}\rangle_{H^\sigma}|\\
&\quad\quad+  |\langle ((K^{\alpha}*v^{\delta}-v^{\delta}) \cdot \nabla) v^{\delta},w^{\delta} \rangle_{H^\sigma}|)\\
&\leq C \|w^{\delta}\|_{H^\sigma} \| v^{\delta, \alpha}\|_{H^{\sigma+1}} \|w^{\delta}\|_{H^\sigma} + C \|K^\alpha*v^{\delta}\|_{H^{\sigma+1}} \|w^{\delta}\|_{H^\sigma}^2\\ 
&+C\|K^{\alpha}*v^{\delta}-v^{\delta}\|_{H^\sigma} \|v^{\delta}\|_{H^{\sigma+1}}\|w^{\delta}\|_{H^\sigma}. 
\end{aligned}
\end{equation*}
As $\sigma+1<s$ and all relevant $H^s$ norms in this estimate are bounded, we can use Gronwall to obtain  
\begin{equation}
\label{eq: Besov l}
     \|w^{\delta}(t) \|_{H^{\sigma}} \leq C(  \|w^{\delta}(0) \|_{H^{\sigma}} + \sup_{\tau\in[0,T_0]}\|(K^{\alpha}*v^{\delta}-v^{\delta})(\tau)\|_{H^{\sigma}}).
\end{equation}

Consider now the specific kernel~\eqref{eq: alpha kernel}. Then, by Lemma \ref{lemma: relation norms}, \eqref{eq: phi leq K}, \eqref{eq: abc} and as $\sigma < s-1$,
\begin{equation*}
\begin{aligned}
 \|(K^{\alpha}*v^{\delta}-v^{\delta})(t)\|_{H^{\sigma}} &= 
\alpha^2 \|\Delta u^{\delta} (t)\|_{H^{\sigma}} \leq  \alpha^2 \|u^{\delta}(t) \|_{H^{\sigma+2}} \leq \alpha^2 \|u^{\delta} (t)\|_{H^{s+1}} \\
&\leq  \frac{\alpha^2}{\alpha} \|v^{\delta} (t)\|_{H^s} \leq  \alpha C,
\end{aligned}
\end{equation*}
$$ \|(K^{\alpha}*v^{\delta}-v^{\delta})(t)\|_{H^s}   = 
\alpha^2 \|\Delta u^{\delta} (t) \|_{H^s} \leq  \alpha^2 \|u^{\delta} (t) \|_{H^{s+2}} \leq  \frac{\alpha^2}{\alpha} \|v^{\delta}(t)\|_{H^{s+1}} = \alpha \frac{C}{\delta}.$$
Consequently,
$$ \|v^{\delta, \alpha}-v^{\delta}\|_{L^{\infty}(0,T_0;H^s)} \leq C \left(\|v^{\alpha}_0-v^{}_0\|_{H^s} + \frac{\|v^{\delta, \alpha}(0)-v^{\delta}(0) \|_{H^{\sigma}} +\alpha}{\delta} T_0 +\frac{\alpha}{\delta} T_0\right).$$
Putting everything together, we derived in (a)-(c)
\begin{equation*}
\begin{aligned} 
&\|v^{\alpha}-v^{}\|_{L^{\infty}(0,T_0;H^s)}\\ 
&\leq  C \left(\|v^{\alpha}_0-v^{}_0\|_{H^s} + \frac{\|v^{\alpha}(0)-v^{}(0) \|_{H^{\sigma}} +\alpha}{\delta} T_0 +\frac{\alpha}{\delta} T_0 + \|v^{\delta}_0-v^{}_0\|_{H^s} + \delta^{s-\sigma-1} T_0\right).
\end{aligned}
\end{equation*}

Given $\epsilon>0$, we take first $\delta$ so small that the last two expressions together are less than $\frac\epsilon2$, and then $\alpha$ (depending on $\delta$) sufficiently small so that the remaining terms become less than $\frac\epsilon2$. It is easy to see that the same conclusions can be drawn for any kernel satisfying the assumptions of Corollary~\ref{corollary: main2}.\\

Let the solution of the Euler equations exist on $[0, T^*)$. It remains to show that the solutions of the Leray-$\alpha$ equations exist as long as the Euler solutions and hence $\|v^{\alpha}-v^{}\|_{L^{\infty}(0,T;H^s)} \to 0 $ as $\alpha \to 0$ for all $0<T<T^*$. To this end, we iterate the argument by solving the Leray-$\alpha$ equations for initial data $v^{\alpha}(T_k)$ for $\alpha \leq \operatorname{min}\{\alpha_0, ..., \alpha_k\}$ and where $\alpha_k$ such that $\|v^{\alpha}(T_k)\|_{H^s} \leq \|v(T_k)\|_{H^s} + \|v_0^{\alpha^*}\|_{H^s}$ for all $\alpha \leq \alpha_k$ (cf. equation \eqref{eq: v_0^a indep of alpha}). We can continue the solution of the Leray-$\alpha$ equations up to any $T<T^*$ in a finite number of iterations and we can define $\overline{\alpha}$ as $\overline{\alpha}:=\operatorname{min}\{\alpha_0, ..., \alpha_{\overline{k}}\}$ where $\overline{k}$ is the number of iterations needed.\\

To show the convergence rate in $H^{s'}, 0\leq s'<s$, let us assume that $v,\ v^{\alpha}$ are the solutions of the Euler resp. Leray-$\alpha$ system on some mutual time interval $[0,T]$ with
\begin{equation}
\label{eq: Abschätzung v}
    \|v\|_{L^{\infty}(0,T;H^s)}, \|v^{\alpha}\|_{L^{\infty}(0,T;H^s)} \leq C(T; \|v_0\|_{H^s}),
\end{equation}
which is justified from the above arguments. Defining $w^{}:=v^{\alpha}-v^{}$,
$$ \partial_t w^{}+ ((K^{\alpha}*v^{\alpha} - v^{\alpha})\cdot \nabla) v^{\alpha} + (w^{} \cdot \nabla) v^{\alpha} + (v^{} \cdot \nabla) w^{} = -\nabla (p^{\alpha}-p^{}).$$
With similar calculations as before (using Lemma \ref{lemma: Besov l} in \hyperref[subsection: USeful estimates]{A2} for the case $0\leq s'\leq d/2$), we obtain for $0\leq s' \leq s-1$
$$ \|w(t)\|_{H^{s'}_{}} \leq C\left(\|w(0)\|_{H^{s'}_{}} + \|K^{\alpha}*v^{\alpha} - v^{\alpha}\|_{H^{s'}_{}}t\right).  $$
Now, for $0\leq s' \leq s-1$ and by Lemma \ref{lemma: relation norms} and the Gagliardo-Nirenberg inequality
$$ \|(K^{\alpha}*v^{\alpha}-v^{\alpha})(t)\|_{H^{s'}} = \|\alpha^2 \Delta u^{\alpha}(t)\|_{H^{s'}} \leq \alpha^2  \| u^{\alpha}(t)\|_{H^{s'+2}} \leq  \alpha^{\iota} C$$
where $C=C(\|v_0\|_{H^s}, T)$ and $\iota=s-s'$ for $s-2\leq s' \leq s-1$ and $\iota=2$ for $0\leq s'\leq s-2$. A convergence rate for $s' \in (s-1,s)$ can be calculated by interpolation or by regularisation of the initial data as before.
\end{proof}

\section{Convergence of weak solutions with local scaling property}
\label{subsection: scaling limit}
In this section, we show Theorem \ref{thm: 2}, i.e.\ the convergence of weak solutions of the Leray-$\alpha$ equations satisfying the scaling property \eqref{eq: scaling prop} to weak solutions of the Euler equations on a smooth bounded domain $\Omega \subset \mathbb{R}^d$, $d\in \{2,3\}$. The conditions are considered far away from the boundaries so that problems with expected boundary layers are avoided.
\begin{proof}[Proof of Theorem 2]\ \\
We follow the idea of the proof in \cite{Constantin.2018}: Let $j(z)$ be a nonnegative smooth function supported in $1<|z|<2$ with $j(z)=j(-z)$ and $\int_{\mathbb{R}^d} j(z) dz = 1$. For a fixed compact set $K \subset \subset \Omega$, for a function $f\in L^1_{loc}(\Omega)$, $x\in K$ and $0<2r<\operatorname{dist}(K,\partial\Omega)$, we define
\begin{equation}
    \label{v_r}
    f_r(x) = \int_{1\leq|z|\leq2} f(x-rz)j(z) \ dz.
\end{equation}
To prove our assertion it suffices to show
\begin{equation}
\label{to prove N}
    \int_0^T \int_{\Omega} (u^{\alpha} \otimes v^{\alpha}) : \nabla \Phi \ dx \ dt \to \int_0^T \int_{\Omega} (v^{\infty} \otimes v^{\infty}) : \nabla \Phi \ dx \ dt \ \ \ \text{ for } \alpha \to 0
\end{equation}
for any fixed, divergence-free test function $\Phi \in C^{\infty}_c(\Omega \times [0,T])$. The linear terms converge by the assumption that $v^{\alpha}(t)$ converge weakly in $L^2(\Omega)$ to $v^{\infty}(t)$ for almost all $t\in (0,T)$.

In order to prove \eqref{to prove N}, we show that for every $\epsilon>0$ there exist $\delta_1,\delta_2>0$ such that for all $\alpha \leq \delta_1$ and $r\leq \delta_2 <\frac{1}{2} \operatorname{dist}(K,\partial\Omega)$,
\begin{equation*}
\begin{aligned}
 &\left| \int_0^T \int_{\Omega} (u^{\alpha} \otimes v^{\alpha}) : \nabla \Phi \ dx \ dt - \int_0^T \int_{\Omega} (v^{\infty} \otimes v^{\infty}) : \nabla \Phi \ dx \ dt \right| \\
 &\leq \left| \int_0^T \int_{\Omega} (u^{\alpha} \otimes v^{\alpha}) : \nabla \Phi \ dx \ dt - \int_0^T \int_{\Omega} \left(u^{\alpha}\otimes (v^{\alpha})_r\right) : \nabla \Phi \ dx \ dt \right| \\
&+ \left| \int_0^T \int_{\Omega} \left(u^{\alpha} \otimes (v^{\alpha})_r\right) : \nabla \Phi \ dx \ dt - \ \int_0^T \int_{\Omega} \left(v^{\infty} \otimes (v^{\infty})_r\right) : \nabla \Phi \ dx \ dt \right|\\
& + \left| \int_0^T \int_{\Omega} \left(v^{\infty} \otimes (v^{\infty})_r\right) : \nabla \Phi \ dx \ dt - \int_0^T \int_{\Omega} (v^{\infty} \otimes v^{\infty}) : \nabla \Phi \ dx \ dt \right|\\
& =: A + B+C < \epsilon
\end{aligned}
\end{equation*}
for a fixed compact set $K$ such that $\operatorname{supp}(\Phi(t)) \subsetneq K$ for all $t \in [0,T]$. We are going to treat the three different parts $A$ to $C$ separately. 

\subsection*{A}

\begin{equation*}
\begin{aligned}
 A&= \left| \int_0^T \int_{\Omega} (u^{\alpha} \otimes v^{\alpha}) : \nabla \Phi \ dx \ dt - \int_0^T \int_{\Omega} \left(u^{\alpha} \otimes (v^{\alpha})_r\right) : \nabla \Phi \ dx \ dt \right| \\
& = \left| \int_0^T \int_{K} u^{\alpha} \otimes \left(v^{\alpha}-(v^{\alpha})_r\right) : \nabla \Phi \ dx \ dt \right| \\
&\leq \|\nabla \Phi \|_{L^{\infty}(0,T;L^{\infty}(K))} \ \left| \int_0^T \int_{K} u^{\alpha}\otimes \left(v^{\alpha}-(v^{\alpha})_r\right) \ dx \ dt \right|.
\end{aligned}
\end{equation*}
Now, for $i,j \in \{1, ...,d\}$,
\begin{equation*}
\begin{aligned}
& \left| \int_0^T \int_{K} u^{\alpha}_i  (v^{\alpha}_j-(v^{\alpha}_j)_{r}) \ dx \ dt \right|\\
& = \left| \int_0^T \int_{K} u^{\alpha}_i(x,t)  \left[ \int_{1\leq |z|\leq 2} \Big( v^{\alpha}_j(x,t)-v^{\alpha}_j (x-rz,t)\Big) j(z) \ dz  \right] \ dx \ dt \right| \\
& \leq \left( \int_{1\leq |z|\leq 2} |j(z)| \left(\int_0^T \left(\int_{K} |u^{\alpha}_i(x,t)|^2  \ dx\right)^{\frac{1}{2}}  \   \left(\int_{K} |v^{\alpha}_j(x,t)-v^{\alpha}_j (x-rz,t)|^2 \ dx\right)^{\frac{1}{2}}  \ dt \right)^2 \ dz \right)^{\frac{1}{2}}\\
& \leq \| u^{\alpha} \|_{L^{\infty}(0,T;L^2(K))}  \left( \int_{1\leq |z|\leq 2} |j(z)| \int_0^T   \int_{K} |v^{\alpha}_j(x,t)-v^{\alpha}_j (x-rz,t)|^2 \ dx  \ dt  \ dz \right)^{\frac{1}{2}}\\
& \leq \| u^{\alpha} \|_{L^{\infty}(0,T;L^2(K))}  \left( \int_{1\leq |z|\leq 2} |j(z)| \ E_K |rz|^{2\gamma}  \ dz \right)^{\frac{1}{2}} \leq C \| u^{\alpha} \|_{L^{\infty}(0,T;L^2(K))}   E_K^{\frac{1}{2}}  |r|^{\gamma}  
\end{aligned}
\end{equation*}
using the Fubini-Tonelli theorem, Jensen's inequality, Hölder's inequality and assumption \eqref{eq: scaling prop} for $\alpha$ small enough such that $\eta(\alpha) < r$. Since $v^{\alpha}=u^{\alpha}-\alpha^2 \Delta u^{\alpha}$ and consequently $\|u^{\alpha}(t)\|_{L^2(\Omega)} \leq C  \|v^{\alpha}(t)\|_{L^2(\Omega)}$, we obtain
$$ \| u^{\alpha} \|_{L^{\infty}(0,T;L^2(K))} \leq \| u^{\alpha} \|_{L^{\infty}(0,T;L^2(\Omega))} \leq C\|v^{\alpha} \|_{L^{\infty}(0,T;L^2(\Omega))} \leq E.$$
We conclude
$$ A = \left| \int_0^T \int_{\Omega} (u^{\alpha} \otimes v^{\alpha}) : \nabla \Phi \ dx \ dt - \int_0^T \int_{\Omega} \left(u^{\alpha} \otimes (v^{\alpha})_r\right) : \nabla \Phi \ dx \ dt \right| \leq C_{\Phi}   E  E_K^{\frac{1}{2}} |r|^{\gamma}$$
for $\alpha$ small enough such that $\eta(\alpha) < r$ and with $2r<\operatorname{dist}(K,\partial\Omega)$.

\subsection*{B}
$$ B = \left| \int_0^T \int_{\Omega} (u^{\alpha} \otimes (v^{\alpha})_r) : \nabla \Phi \ dx \ dt - \ \int_0^T \int_{\Omega} (v^{\infty} \otimes (v^{\infty})_r) : \nabla \Phi \ dx \ dt \right|$$
$$ = \left| \int_0^T \int_{K} \left(u^{\alpha} \otimes [(v^{\alpha})_r-(v^{\infty})_r]\right) : \nabla \Phi \ dx \ dt + \ \int_0^T \int_{K} ((u^{\alpha}-v^{\infty}) \otimes (v^{\infty})_r) : \nabla \Phi \ dx \ dt \right|. $$

First, we argue as in \cite{Constantin.2018} and observe that $(v^{\alpha}(t))_r(x)$ converges pointwise to $(v^{\infty}(t))_r(x)$ in $K$ at fixed $r$ since $v^{\alpha}(t)$ converges weakly to $v^{\infty}(t)$ in $L^2(\Omega)$ for a.e.\ $t\in (0,T)$. Indeed, we can write
$$ (v^{\alpha}(t))_r(x)=\int_{1\leq|z|\leq2} v^{\alpha}\left(t,x-rz\right) \ j(z) \ dz = \int_{\{x\}-rA} v^{\alpha}(t,y) \ j\left(\frac{x-y}{r}\right) \  r^{-d} \ dy$$
where $A:=\{z\in \mathbb{R}^d \ | \ 1\leq |z|\leq 2\}.$ Now,
$$ |(v^{\alpha}(t))_r(x)| \leq \int_{\{x\}-rA} \left| v^{\alpha}(t,y) \ j\left(\frac{x-y}{r}\right) \  r^{-d}\right| \ dy $$
$$ \leq \|v^{\alpha}(t)\|_{L^2(\Omega)} \left(\int_{\{x\}-rA} \left( j\left(\frac{x-y}{r}\right) \  r^{-d}\right)^2 \ dy\right)^{\frac{1}{2}} $$ 
$$= \|v^{\alpha}(t)\|_{L^2(\Omega)} \left(\int_{1\leq |z|\leq 2} j(z)^2 \ \  r^{-2d} \  r^d \ dz\right)^{\frac{1}{2}} $$
$$ = r^{-d/2} \ \|v^{\alpha}(t)\|_{L^2(\Omega)} \ \|j\|_{L^2(\mathbb{R}^d)} \leq C  r^{-d/2} \ \|v^{\alpha}(t)\|_{L^2(\Omega)}. $$
By means of the dominated convergence theorem, $(v^{\alpha}(t))_r \to (v^{\infty}(t))_r$ in $L^2$ as $\alpha \to 0$ for a.e. $t\in (0,T)$.

Secondly, it remains to show that $u^{\alpha}(t) \rightharpoonup v^{\infty}(t)$ in $L^2(K)$ for almost all $t\in (0,T)$.
Now, by elliptic regularity theory (\cite[Theorem 8.12]{Gilbarg.1983}) and the fact that $\|u^{\alpha}(t)\|_{L^2(\Omega)} \leq C  \|v^{\alpha}(t)\|_{L^2(\Omega)}$, we obtain $\|u^{\alpha}(t)\|_{H^2(\Omega)} \leq \frac{C}{\alpha^2}  \|v^{\alpha}(t)\|_{L^2(\Omega)}$ and  by interpolation $\|u^{\alpha}(t)\|_{H^1(\Omega)} \leq \frac{C}{\alpha}  \|v^{\alpha}(t)\|_{L^2(\Omega)}$. Thus,
$$ \|\alpha^2 \Delta u^{\alpha}(t)\|_{L^2(\Omega)} \leq  \alpha^2 \|u^{\alpha}(t)\|_{H^2(\Omega)} \leq C   \|v^{\alpha}\|_{L^{\infty}(0,T;L^2(\Omega))} \leq CE $$
and
$$ \|\alpha^2 \Delta u^{\alpha}(t)\|_{H^{-1}(\Omega)} \leq \alpha^2 \|u^{\alpha}(t)\|_{H^1_0(\Omega)} \leq C \alpha  \|v^{\alpha}\|_{L^{\infty}(0,T;L^2(\Omega))} \leq CE \alpha \ \to 0 \ \ \text{ as } \alpha \to 0.$$
Hence,
$$ \alpha^2 \Delta u^{\alpha}(t) \to 0 \ \ \ \text{ strongly in } H^{-1}(\Omega)$$
and consequently
$$ \alpha^2 \Delta u^{\alpha}(t) \rightharpoonup 0 \ \ \ \text{ weakly in } L^2(\Omega).$$

\subsection*{C}
$$ C= \left| \int_0^T \int_{\Omega} \left(v^{\infty} \otimes (v^{\infty})_r\right) : \nabla \Phi \ dx \ dt - \int_0^T \int_{\Omega} (v^{\infty} \otimes v^{\infty}) : \nabla \Phi \ dx \ dt \right|.$$
It is enough to show that
$$ \|(v^{\infty})_r-v^{\infty}\|_{L^2(0,T;L^2(K))} \to 0 \ \ \text{ as } r\to 0,$$
but this follows from the fact that translations are strongly continuous in $L^2$.\\

All together, we can find a sequence $(\alpha_n)_n$, depending on $r_n$, and $(r_n)_n$ small enough such that for every $\epsilon>0$ we can choose $\alpha_n, r_n$ small enough such that
$$ \left| \int_0^T \int_{\Omega} (u^{\alpha} \otimes v^{\alpha}) : \nabla \Phi \ dx \ dt - \int_0^T \int_{\Omega} (v^{\infty} \otimes v^{\infty}) : \nabla \Phi \ dx \ dt \right| < \epsilon.$$

\end{proof}

\section*{Appendix}
\addcontentsline{toc}{section}{Appendix}

\subsection*{A.1 Proof of Lemma \ref{lemma: relation norms}}
\label{subsection: relation between norms}
 \begin{proof}[Proof of Lemma \ref{lemma: relation norms}]\ \\
 (a)
$$\|u\|_{H^s}^2 =  \int_{\mathbb{R}^n} (1+|\xi|^2)^{s}\left(\frac{1}{1+\alpha^2|\xi|^2}\right)^2\ |\hat{v}^{}(\xi)|^2 \ d\xi \leq  \int_{\mathbb{R}^n} (1+|\xi|^2)^{s} \ |\hat{v}^{}(\xi)|^2 \ d\xi = \|v\|_{H^s}^2. $$
(b)
\begin{equation*}
\begin{aligned}
&\alpha^2\|u\|_{H^{s+1}}^2  = \int_{\mathbb{R}^n} \alpha^2(1+|\xi|^2)^{s+1}\left(\frac{1}{1+\alpha^2|\xi|^2}\right)^2\ |\hat{v}^{}(\xi)|^2 \ d\xi \\
&=    \int_{\mathbb{R}^n} \frac{\alpha^2+ \alpha^2|\xi|^2}{1+2\alpha^2|\xi|^2 + \alpha^4|\xi|^4}\ (1+|\xi|^2)^{s} |\hat{v}^{}(\xi)|^2 \ d\xi \leq \int_{\mathbb{R}^n}  (1+|\xi|^2)^{s} |\hat{v}^{}(\xi)|^2 \ d\xi = \|v\|_{H^s}^2.
\end{aligned}
\end{equation*}
(c)
\begin{equation*}
\begin{aligned}
&\alpha^4\|u\|_{H^{s+2}}^2 = \int_{\mathbb{R}^n} \alpha^4(1+|\xi|^2)^{s+2}\left(\frac{1}{1+\alpha^2|\xi|^2}\right)^2\ |\hat{v}^{}(\xi)|^2 \ d\xi \\
&=  \int_{\mathbb{R}^n} \frac{(\alpha^2+ \alpha^2|\xi|^2)^2}{(1+\alpha^2|\xi|^2)^2}\ (1+|\xi|^2)^{s} |\hat{v}^{}(\xi)|^2 \ d\xi \leq \int_{\mathbb{R}^n}  (1+|\xi|^2)^{s} |\hat{v}^{}(\xi)|^2 \ d\xi = \|v\|_{H^s}^2.
\end{aligned}
\end{equation*}
\end{proof}

\subsection*{A.2 A useful estimate}
\label{subsection: USeful estimates}
\begin{lemma}
    \label{lemma: Besov l}
Let $s>d/2+1$, $d\in\{2,3\}$, $\sigma\in [0,s-1]$, $T>0$.\\ Let $w \in \mathcal{C}([0,T];H^{\sigma}(\mathbb{R}^d)\cap H^1(\mathbb{R}^d))\cap W^{1,1}(0,T;L^2(\mathbb{R}^d))$, $\tilde{w}, h \in \mathcal{C}([0,T];H^{\sigma}(\mathbb{R}^d))$, $v_1,v_2,v_3 \in \mathcal{C}([0,T]H^s(\mathbb{R}^d))$, $p\in \mathcal{C}([0,T];H^1(\mathbb{R}))$ and $w, v_2$ divergence-free such that
\begin{equation}
\label{eq: A4}
    \partial_t w + (\tilde{w} \cdot \nabla) v_1 + (v_2 \cdot \nabla) w +  (h \cdot \nabla) v_3 = -\nabla p.
\end{equation}
Then,
$$ \|w^{}(t)\|_{H^{\sigma}_{}} \leq C \left( \|w^{}(0)\|_{H^{\sigma}_{}} + \int_0^T \| v_1\|_{H^s} \|\tilde{w}^{} \|_{H^{\sigma}} + \|v_2\|_{H^s}  \|w^{} \|_{H^{\sigma}}  \ + \|h\|_{H^{\sigma}} \|v_3\|_{H^s}\ dt\right).$$
\end{lemma}
\begin{proof}
The proof is restricted to the case $\sigma\leq d/2 $. For the case $d/2<\sigma \leq s-1$, we can conclude with \eqref{eq: est 0} and the Banach algebra property of $H^{\sigma}$.

We will use some results from \cite[Chapter 2]{Bahouri.2011}. Let $\Delta_j$ denote the nonhomogeneous dyadic blocks from Littlewood-Paley theory. Application of $\Delta_j$ to \eqref{eq: A4} leads to
$$ \partial_t w^{}_j + \Delta_j (\tilde{w}^{} \cdot \nabla) v_1 +  (v_2 \cdot \nabla) w^{}_j + \Delta_j(h\cdot \nabla) v_3 =  R_j - \nabla p_j,$$
where $w_j:=\Delta_jw$ and $R_j:=[v_2\cdot \nabla,\Delta_j] w^{}$ is a commutator. Multiplying by $w^{}_j$ and integrating over space gives
\begin{align*}
    \int_{\mathbb{R}^d} \partial_t w^{}_j \cdot w^{}_j \ dx \ &+ \int_{\mathbb{R}^d} \left[\Delta_j (\tilde{w} \cdot \nabla) v_1 \ + \ (v_2\cdot \nabla) w^{}_j \ + \ \Delta_j(h\cdot \nabla) v_3 \right] \cdot w^{}_j \ dx \\
    &=  \int_{\mathbb{R}^d} \left[R_j -\nabla p_j\right] \cdot w^{}_j \ dx,
\end{align*} 
hence
$$ \frac{1}{2} \frac{d}{dt} \|w^{}_j\|_{L^2}^2 + \int_{\mathbb{R}^d} \Delta_j (\tilde{w} \cdot \nabla) v_1 \cdot w^{}_j\ dx  + \int_{\mathbb{R}^d}\Delta_j(h\cdot \nabla) v_3 \cdot w^{}_j \ dx \ =  \int_{\mathbb{R}^d} R_j \cdot w^{}_j \ dx, $$
and then
$$ \frac{1}{2} \frac{d}{dt} \|w^{}_j\|_{L^2}^2 \leq  \|\Delta_j (\tilde{w} \cdot \nabla) v_1\|_{L^2} \|w^{}_j\|_{L^2}+ \|\Delta_j(h \cdot \nabla) v_3 \|_{L^2} \|w^{}_j\|_{L^2} + \|R_j\|_{L^2} \|w^{}_j\|_{L^2}. $$
Integrating over time,
$$  \|w^{}_j(T)\|_{L^2} \leq  \int_0^T \|\Delta_j (\tilde{w} \cdot \nabla) v_1\|_{L^2} + \|\Delta_j(h\cdot \nabla) v_3 \|_{L^2}  + \|R_j\|_{L^2} \ dt + \|w^{}_j(0)\|_{L^2},$$
and multiplying by $2^{j\sigma}$ leads to
\begin{align*}
    2^{j\sigma}\|w^{}_j(T)\|_{L^2} \leq & \int_0^T 2^{j\sigma}\left\|\Delta_j (\tilde{w}\cdot \nabla) v_1\right\|_{L^2} + 2^{j\sigma}\|\Delta_j(h\cdot \nabla) v_3\|_{L^2}   + 2^{j\sigma}\|R_j\|_{L^2} \ dt\\
    &+ 2^{j\sigma}\|w^{}_j(0)\|_{L^2}.
\end{align*} 
We take the $\textit{l}^r$-norm with respect to the index $j$ for some $1\leq r  \leq \infty$ and use the triangle inequality to obtain
$$  \left\|2^{j\sigma}\|w^{}_j(T)\|_{L^2}\right\|_{\textit{l}^r} \leq  \left\|\int_0^T 2^{j\sigma}\|\Delta_j (\tilde{w} \cdot \nabla) v_1\|_{L^2}  \ dt\right\|_{\textit{l}^r} \ + \ \left\|\int_0^T 2^{j\sigma}\|\Delta_j(h\cdot \nabla) v_3\|_{L^2}  \ dt\right\|_{\textit{l}^r} $$ 
$$+ \left\|\int_0^T 2^{j\sigma}\|R_j\|_{L^2} \ dt \right\|_{\textit{l}^r}+ \left\|2^{j\sigma}\|w^{}_j(0)\|_{L^2}\right\|_{\textit{l}^r}.$$
Now, by the triangle inequality for the norm $\left\|\cdot\right\|_{\textit{l}^r}$ and the commutator estimate \cite[Lemma 2.100]{Bahouri.2011}, for $0\leq \sigma < 1 +\frac{d}{2}$,
$$\left\|\int_0^T 2^{j\sigma}\|R_j\|_{L^2} \ dt \right\|_{\textit{l}^r} \leq\int_0^T  \left\|2^{j\sigma}\|R_j\|_{L^2} \right\|_{\textit{l}^r}\ dt $$ $$\leq \int_0^T C \left\|\nabla v_2\right\|_{{B^{\frac{d}{2}}_{2,\infty}}\cap L^{\infty}} \|w^{}\|_{B^{\sigma}_{2,r}}   \ dt .$$
Furthermore, for $\sigma \in [0,s-1]$,
$$\left\|\int_0^T 2^{j\sigma}\|\Delta_j (\tilde{w} \cdot \nabla) v_1\|_{L^2}  \ dt\right\|_{\textit{l}^r}\leq \int_0^T \|(\tilde{w} \cdot \nabla) v_1\|_{B^{\sigma}_{2,r}}\ dt \leq \int_0^T \|\tilde{w}\|_{B^{\sigma}_{2,r}} \|v_1\|_{H^s}\ dt, $$
$$ \left\|\int_0^T 2^{j\sigma}\|\Delta_j(h\cdot \nabla) v_3 \|_{L^2}  \ dt\right\|_{\textit{l}^r}  \leq \int_0^T \|(h\cdot \nabla) v_3 \|_{B^{\sigma}_{2,r}}\ dt  \leq \int_0^T \|h \|_{B^{\sigma}_{2,r}} \|v_3\|_{H^s}\ dt. $$
This follows from the fact that $B^{\sigma}_{2,r}$ is a real interpolation space between $L^2$ and $H^{s-1}$ for $1\leq r \leq \infty$, $0< \theta <1$, $ \sigma = \theta (s-1)$ (\cite[Thm. 6.4.5]{Bergh.1976}).

In total, for $r=2$,
\begin{align*} \|w^{}(T)\|_{B^{\sigma}_{2,2}} \leq & \|w^{}(0)\|_{B^{\sigma}_{2,2}}\\
+ &\int_0^T C \left(\| v_2\|_{{H^s}} \|w^{} \|_{B^{\sigma}_{2,2}} + \|v_1\|_{H^s} \|\tilde{w}\|_{B^{\sigma}_{2,2}}    \ + \|h\|_{B^{\sigma}_{2,2}} \|v_3\|_{H^s}\right)\ dt.
\end{align*}
\end{proof}

\textbf{Acknowledgement:} The authors gratefully acknowledge the support by the Research Training Group 2339 IntComSin of Deutsche Forschungsgemeinschaft (DFG, German Research Foundation) -- Project-ID 321821685, and of the DFG Priority Programme 2410 CoScaRa -- Project no.~525716336.

\printbibliography

@article {Feff.2014,
    AUTHOR = {Fefferman, Charles L. and McCormick, David S. and Robinson,
              James C. and Rodrigo, Jose L.},
     TITLE = {Higher order commutator estimates and local existence for the
              non-resistive {MHD} equations and related models},
   JOURNAL = {J. Funct. Anal.},
  FJOURNAL = {Journal of Functional Analysis},
    VOLUME = {267},
      YEAR = {2014},
    NUMBER = {4},
     PAGES = {1035--1056},
      %ISSN = {0022-1236,1096-0783},
   MRCLASS = {35Q35 (76W05)},
  MRNUMBER = {3217057},
MRREVIEWER = {Paolo\ Secchi},
      % DOI = {10.1016/j.jfa.2014.03.021},
       %URL = {https://doi.org/10.1016/j.jfa.2014.03.021},
}

@article{Geurts.2008,
 author = {Geurts, Bernard J. and Kuczaj, Arkadiusz K. and Titi, Edriss S.},
 year = {2008},
 title = {Regularization modeling for large-eddy simulation of homogeneous isotropic decaying turbulence},
 pages = {344008},
 pagination = {page},
 volume = {41},
 journaltitle = {Journal of Physics A: Mathematical and Theoretical},
 shortjournal = {J. Phys. A: Math. Theor.},
 doi = {10.1088/1751-8113/41/34/344008},
 number = {34}
}

@article{Guermond.2003,
 author = {Guermond, J. L. and Oden, J. T. and Prudhomme, S.},
 year = {2003},
 title = {An interpretation of the Navier--Stokes-alpha model as a frame-indifferent Leray regularization},
 pages = {23--30},
 pagination = {page},
 volume = {177},
 journaltitle = {Physica D: Nonlinear Phenomena},
 doi = {10.1016/S0167-2789(02)00748-0},
 number = {1-4}
}

@article{Holm.1999,
 author = {Holm, Darryl D.},
 year = {1999},
 title = {Fluctuation effects on 3D Lagrangian mean and Eulerian mean fluid motion},
 pages = {215--269},
 pagination = {page},
 volume = {133},
 journaltitle = {Physica D: Nonlinear Phenomena},
 doi = {10.1016/S0167-2789(99)00093-7},
 number = {1-4}
}

@article{Holm.1998,
 author = {Holm, Darryl D. and Marsden, JERROLD E. and Ratiu, Tudor S.},
 year = {1998},
 title = {Euler-Poincar{\'e} models of ideal fluids with nonlinear dispersion},
 pages = {4173--4176},
 pagination = {page},
 volume = {80},
 journaltitle = {Physical Review Letters},
 shortjournal = {Phys. Rev. Lett.},
 doi = {10.1103/PhysRevLett.80.4173},
 number = {19}
}

@article{Holm.1998b,
 author = {Holm, Darryl D. and Marsden, JERROLD E. and Ratiu, Tudor S.},
 year = {1998},
 title = {The Euler--Poincar{\'e} equations and semidirect products with applications to continuum theories},
 pages = {1--81},
 pagination = {page},
 volume = {137},
 journaltitle = {Advances in Mathematics},
 doi = {10.1006/aima.1998.1721},
 number = {1}
}

@article{Kato.1972,
 author = {Kato, Tosio},
 year = {1972},
 title = {Nonstationary flows of viscous and ideal fluids in $R^3$},
 pages = {296--305},
 pagination = {page},
 volume = {9},
 journaltitle = {Journal of Functional Analysis},
 doi = {10.1016/0022-1236(72)90003-1},
 number = {3}
}

@article{Kato.1988,
 author = {Kato, Tosio and Ponce, Gustavo},
 year = {1988},
 title = {Commutator estimates and the Euler and Navier--Stokes equations},
 pages = {891--907},
 pagination = {page},
 volume = {41},
 journaltitle = {Communications on Pure and Applied Mathematics},
 shortjournal = {Comm Pure Appl Math},
 doi = {10.1002/cpa.3160410704},
 number = {7}
}

@article{Leray.1934,
 author = {Leray, Jean},
 year = {1934},
 title = {Sur le mouvement d'un liquide visqueux emplissant l'espace},
 pages = {193--248},
 pagination = {page},
 volume = {63},
 journaltitle = {Acta Mathematica},
 shortjournal = {Acta Math.},
 doi = {10.1007/BF02547354},
 number = {0}
}

@article{Linshiz.2010,
 author = {Linshiz, Jasmine S. and Titi, Edriss S.},
 year = {2010},
 title = {On the convergence rate of the Euler-$\alpha$, an inviscid second-grade complex fluid, model to the Euler equations},
 pages = {305--332},
 pagination = {page},
 volume = {138},
 journaltitle = {Journal of Statistical Physics},
 shortjournal = {J Stat Phys},
 doi = {10.1007/s10955-009-9916-9},
 number = {1-3}
}

@article{LopesFilho.2015,
 author = {{Lopes Filho}, Milton C. and {Nussenzveig Lopes}, Helena J. and Titi, Edriss S. and Zang, Aibin},
 year = {2015},
 title = {Convergence of the 2D Euler-$\alpha$ to Euler equations in the Dirichlet case: Indifference to boundary layers},
 pages = {51--61},
 pagination = {page},
 volume = {292-293},
 journaltitle = {Physica D: Nonlinear Phenomena},
 doi = {10.1016/j.physd.2014.11.001}
}

@book{Majda.2010,
 author = {Majda, Andrew J. and Bertozzi, Andrea L.},
 year = {2002},
 title = {Vorticity and Incompressible Flow},
 publisher = {{Cambridge University Press}},
 %isbn = {9780521630573},
 abstract = {}
}

@article{Marsden.2000,
 author = {Marsden, J. E. and Ratiu, T. S. and Shkoller, S.},
 year = {2000},
 title = {The geometry and analysis of the averaged Euler equations and a new diffeomorphism group},
 pages = {582--599},
 pagination = {page},
 volume = {10},
 journaltitle = {Geometric and Functional Analysis},
 shortjournal = {GAFA, Geom. funct. anal.},
 doi = {10.1007/PL00001631},
 number = {3}
}

@article{MARSDEN.2003,
 author = {Marsden, JERROLD E. and Shkoller, Steve},
 year = {2003},
 title = {The anisotropic Lagrangian averaged Euler and Navier-Stokes equations},
 pages = {27--46},
 pagination = {page},
 volume = {166},
 journaltitle = {Archive for Rational Mechanics and Analysis},
 doi = {10.1007/s00205-002-0207-8},
 number = {1}
}

@article{Masmoudi.2007,
 author = {Masmoudi, Nader},
 year = {2007},
 title = {Remarks about the inviscid limit of the Navier--Stokes system},
 pages = {777--788},
 pagination = {page},
 volume = {270},
 journaltitle = {Communications in Mathematical Physics},
 shortjournal = {Commun. Math. Phys.},
 doi = {10.1007/s00220-006-0171-5},
 number = {3}
}

@article{Oliver.2001,
 author = {Oliver, Marcel and Shkoller, Steve},
 year = {2001},
 title = {The vortex blob method as a second-grade non-Newtonian fluid},
 pages = {295--314},
 pagination = {page},
 volume = {26},
 journaltitle = {Communications in Partial Differential Equations},
 doi = {10.1081/PDE-100001756},
 number = {1-2}
}

@article{PietarilaGraham.2008,
 author = {{Pietarila Graham}, Jonathan and Holm, Darryl D. and Mininni, Pablo D. and Pouquet, Annick},
 year = {2008},
 title = {Three regularization models of the Navier--Stokes equations},
 volume = {20},
 journaltitle = {Physics of Fluids},
 doi = {10.1063/1.2880275},
 number = {3}
}

@article{Shkoller.2000,
 author = {Shkoller, Steve},
 year = {2000},
 title = {Analysis on Groups of Diffeomorphisms of Manifolds with Boundary and the Averaged Motion of a Fluid},
 volume = {55},
 journaltitle = {Journal of Differential Geometry},
 shortjournal = {J. Differential Geom.},
 doi = {10.4310/jdg/1090340568},
 number = {1}
}

@article{V.Chepyzhov.2007,
 author = {Chepyzhov, Vladimir V. and Titi, Edriss S. and  Vishik, Mark I.},
 year = {2007},
 title = {On the convergence of solutions of the Leray-$\alpha$ model to the trajectory attractor of the 3D Navier-Stokes system},
 pages = {481--500},
 pagination = {page},
 volume = {17},
 journaltitle = {Discrete {\&} Continuous Dynamical Systems - A},
 doi = {10.3934/dcds.2007.17.481},
 number = {3}
}

@article{Vorotnikov.2012,
 author = {Vorotnikov, Dmitry},
 year = {2012},
 title = {Global generalized solutions for Maxwell-alpha and Euler-alpha equations},
 pages = {309--327},
 pagination = {page},
 volume = {25},
 journaltitle = {Nonlinearity},
 doi = {10.1088/0951-7715/25/2/309},
 number = {2}
}

@article{Geurts.2006,
 author = {Geurts, Bernard J. and Holm, Darryl D.},
 year = {2006},
 title = {Leray and LANS-$\alpha$ modelling of turbulent mixing},
 pages = {N10},
 pagination = {page},
 volume = {7},
 journaltitle = {Journal of Turbulence},
 doi = {10.1080/14685240500501601}
}

@article{Dunn.1974,
 author = {Dunn, J. Ernest and Fosdick, Roger L.},
 year = {1974},
 title = {Thermodynamics, stability, and boundedness of fluids of complexity 2 and fluids of second grade},
 pages = {191--252},
 pagination = {page},
 volume = {56},
 journaltitle = {Archive for Rational Mechanics and Analysis},
 doi = {10.1007/BF00280970},
 number = {3}
}

@article{Abbate.2024,
 author = {Abbate, Stefano and Crippa, Gianluca and Spirito, Stefano},
 year = {2024},
 title = {Strong convergence of the vorticity and conservation of the energy for the $\alpha$-Euler equations},
 pages = {035012},
 pagination = {page},
 volume = {37},
 journaltitle = {Nonlinearity},
 doi = {10.1088/1361-6544/ad1cdf},
 number = {3}
}

@book{Bahouri.2011,
 author = {Bahouri, Hajer and Chemin, Jean-Yves and Danchin, Rapha{\"e}l},
 year = {2011},
 title = {Fourier Analysis and Nonlinear Partial Differential Equations},
 volume = {343},
 publisher = {{Springer Berlin Heidelberg}},
 isbn = {978-3-642-16829-1},
 location = {Berlin, Heidelberg},
 doi = {10.1007/978-3-642-16830-7}
}

@article{Bardos.2010,
 author = {Bardos, Claude and Linshiz, Jasmine S. and Titi, Edriss S.},
 year = {2010},
 title = {Global regularity and convergence of a Birkhoff--Rott-$\alpha$ approximation of the dynamics of vortex sheets of the two--dimensional Euler equations},
 pages = {697--746},
 pagination = {page},
 volume = {63},
 journaltitle = {Communications on Pure and Applied Mathematics},
 shortjournal = {Comm Pure Appl Math},
 doi = {10.1002/CPA.20305},
 number = {6}
}

@article{Berselli.2012,
 author = {Berselli, Luigi C. and Bisconti, Luca},
 year = {2012},
 title = {On the structural stability of the Euler--Voigt and Navier--Stokes--Voigt models},
 pages = {117--130},
 pagination = {page},
 volume = {75},
 journaltitle = {Nonlinear Analysis: Theory, Methods {\&} Applications},
 doi = {10.1016/j.na.2011.08.011},
 number = {1}
}

@article{Boutros.2023,
 author = {Boutros, Daniel W. and Titi, Edriss S.},
 year = {2023},
 title = {Onsager's conjecture for subgrid scale $\alpha$-models of turbulence},
 pages = {133553},
 pagination = {page},
 volume = {443},
 journaltitle = {Physica D: Nonlinear Phenomena},
 doi = {10.1016/j.physd.2022.133553}
}

@article{Busuioc.2012,
 author = {Busuioc, A. V. and Iftimie, D. and {Lopes Filho}, M. C. and {Nussenzveig Lopes}, H. J.},
 year = {2012},
 title = {Incompressible Euler as a limit of complex fluid models with Navier boundary conditions},
 pages = {624--640},
 pagination = {page},
 volume = {252},
 journaltitle = {Journal of Differential Equations},
 doi = {10.1016/j.jde.2011.06.007},
 number = {1},
 pagetotal = {17}
}

@article{Busuioc.2020,
 author = {Busuioc, Adriana V. and Iftimie, Dragos and {Lopes Filho}, Milton D. and {Nussenzveig Lopes}, Helena J.},
 year = {2020},
 title = {The limit $\alpha \to 0$ of the $\alpha$-Euler equations in the half-plane with no-slip boundary conditions and vortex sheet initial data},
 pages = {5257--5286},
 pagination = {page},
 volume = {52},
 journaltitle = {SIAM Journal on Mathematical Analysis},
 shortjournal = {SIAM J. Math. Anal.},
 doi = {10.1137/19M1303721},
 number = {5}
}

@article{Busuioc.2017,
 author = {Busuioc, Adriana Valentina and Iftimie, Drago{\c{s}}},
 year = {2017},
 title = {Weak solutions for the $\alpha$-Euler equations and convergence to Euler},
 pages = {4534--4557},
 pagination = {page},
 volume = {30},
 journaltitle = {Nonlinearity},
 doi = {10.1088/1361-6544/aa8982},
 number = {12}
}

@article{Busuioc.2003,
 author = {Busuioc, Adriana Valentina and Ratiu, Tudor S.},
 year = {2003},
 title = {The second grade fluid and averaged Euler equations with Navier-slip boundary conditions},
 pages = {1119--1149},
 pagination = {page},
 volume = {16},
 journaltitle = {Nonlinearity},
 doi = {10.1088/0951-7715/16/3/318},
 number = {3}
}

@article{Busuioc.1999,
 author = {Busuioc, Valentina},
 year = {1999},
 title = {On second grade fluids with vanishing viscosity},
 pages = {1241--1246},
 pagination = {page},
 volume = {328},
 journaltitle = {Comptes Rendus de l'Acad{\'e}mie des Sciences - Series I - Mathematics},
 doi = {10.1016/S0764-4442(99)80447-9},
 number = {12}
}

@article{Cao.2009,
 author = {Cao, Yanping and Titi, Edriss S.},
 year = {2009},
 title = {On the rate of convergence of the two-dimensional $\alpha$-models of turbulence to the Navier--Stokes equations},
 pages = {1231--1271},
 pagination = {page},
 volume = {30},
 journaltitle = {Numerical Functional Analysis and Optimization},
 doi = {10.1080/01630560903439189},
 number = {11-12}
}

@book{Bergh.1976,
 author = {Bergh, Jöran and Löfström, Jörgen},
 year = {1976},
 title = {Interpolation Spaces: An Introduction},
 volume = {223},
 publisher = {{Springer Berlin Heidelberg}},
 location = {Berlin, Heidelberg},
 doi = {10.1007/978-3-642-66451-9}
}

@article{Cheskidov.2005,
 author = {Cheskidov, Alexey and Holm, Darryl D. and Olson, Eric and Titi, Edriss S.},
 year = {2005},
 title = {On a Leray--$\alpha$ model of turbulence},
 pages = {629--649},
 pagination = {page},
 volume = {461},
 journaltitle = {Proceedings of the Royal Society A: Mathematical, Physical and Engineering Sciences},
 shortjournal = {Proc. R. Soc. A.},
 doi = {10.1098/rspa.2004.1373},
 number = {2055}
}

@article{Constantin.2018,
 author = {Constantin, Peter and Vicol, Vlad},
 year = {2018},
 title = {Remarks on high Reynolds numbers hydrodynamics and the inviscid limit},
 pages = {711--724},
 pagination = {page},
 volume = {28},
 journaltitle = {Journal of Nonlinear Science},
 shortjournal = {J Nonlinear Sci},
 doi = {10.1007/s00332-017-9424-z},
 number = {2}
}

@article{DaBeiraoVeiga.1994,
 author = {Beir{\~a}o Da Veiga, H.},
 year = {1994},
 title = {Singular limits in compressible fluid dynamics},
 pages = {313--327},
 pagination = {page},
 volume = {128},
 journaltitle = {Archive for Rational Mechanics and Analysis},
 doi = {10.1007/BF00387711},
 number = {4}
}

@article{DaBeiraoVeiga.1994b,
 author = { Beir{\~a}o Da Veiga, H.},
 year = {1994},
 title = {On the singular limit for slightly compressible fluids},
 pages = {205--218},
 pagination = {page},
 volume = {2},
 journaltitle = {Calculus of Variations and Partial Differential Equations},
 shortjournal = {Calc. Var},
 doi = {10.1007/BF01191342},
 number = {2}
}

@book{Bedrossian.2022,
 author = {Bedrossian, Jacob and Vicol, Vlad},
 year = {2022},
 title = {The Mathematical Analysis of the Incompressible Euler and Navier--Stokes Equations: An Introduction},
 publisher = {Graduate Studies in Mathematics 225, American Mathematical Society},
 volume = {225},
}

@book{Gilbarg.1983,
 author = {Gilbarg, David and Trudinger, Neil S.},
 year = {1983},
 title = {Elliptic Partial Differential Equations of Second Order},
 publisher = {{Springer Berlin Heidelberg}},
 location = {Berlin, Heidelberg},
}

@book{Frisch.1995,
 author = {Frisch, Uriel},
 year = {1995},
 title = {Turbulence: The Legacy of A. N. Kolmogorov. },
publisher = {{Cambridge University Press}},
location = {Cambridge}
}

@article {MR3345360,
    AUTHOR = {Lopes Filho, Milton C. and Nussenzveig Lopes, Helena J. and
              Titi, Edriss S. and Zang, Aibin},
     TITLE = {Approximation of 2{D} {E}uler equations by the second-grade
              fluid equations with {D}irichlet boundary conditions},
   JOURNAL = {J. Math. Fluid Mech.},
  FJOURNAL = {Journal of Mathematical Fluid Mechanics},
    VOLUME = {17},
      YEAR = {2015},
    NUMBER = {2},
     PAGES = {327--340},
   MRCLASS = {35Q31 (35B25 76A10 76D05)},
  MRNUMBER = {3345360},

}

@article {Drivas,
    AUTHOR = {Drivas, Theodore D. and Nguyen, Huy Q.},
     TITLE = {Remarks on the emergence of weak {E}uler solutions in the
              vanishing viscosity limit},
   JOURNAL = {J. Nonlinear Sci.},
  FJOURNAL = {Journal of Nonlinear Science},
    VOLUME = {29},
      YEAR = {2019},
    NUMBER = {2},
     PAGES = {709--721},
   MRCLASS = {76D03 (35B25 35Q31 76D05)},
  MRNUMBER = {3927111},
MRREVIEWER = {Carlo-Romano\ Grisanti},
}

@article {wiedemann2011,
    AUTHOR = {Wiedemann, Emil},
     TITLE = {Existence of weak solutions for the incompressible {E}uler
              equations},
   JOURNAL = {Ann. Inst. H. Poincar\'e{} C Anal. Non Lin\'eaire},
  FJOURNAL = {Annales de l'Institut Henri Poincar\'e{} C. Analyse Non
              Lin\'eaire},
    VOLUME = {28},
      YEAR = {2011},
    NUMBER = {5},
     PAGES = {727--730},
      %ISSN = {0294-1449,1873-1430},
   MRCLASS = {35Q31 (35A01 35D30 76B03)},
  MRNUMBER = {2838398},
MRREVIEWER = {Franck\ Sueur},
      % DOI = {10.1016/j.anihpc.2011.05.002},
       %URL = {https://doi.org/10.1016/j.anihpc.2011.05.002},
}

@article {MR1199199,
    AUTHOR = {Beir\~ao da Veiga, H.},
     TITLE = {Perturbation theorems for linear hyperbolic mixed problems and
              applications to the compressible {E}uler equations},
   JOURNAL = {Comm. Pure Appl. Math.},
  FJOURNAL = {Communications on Pure and Applied Mathematics},
    VOLUME = {46},
      YEAR = {1993},
    NUMBER = {2},
     PAGES = {221--259},
     % ISSN = {0010-3640,1097-0312},
   MRCLASS = {35L50 (35L60 35Q35 76N10)},
  MRNUMBER = {1199199},
MRREVIEWER = {Paolo\ Secchi},
     %  DOI = {10.1002/cpa.3160460206},
      % URL = {https://doi.org/10.1002/cpa.3160460206},
}

\end{document}